\documentclass[english]{amsart}
\usepackage{lmodern}
\usepackage[T1]{fontenc}
\usepackage[latin9]{inputenc}
\usepackage{amsthm}
\usepackage{amssymb}
\usepackage{microtype}

\makeatletter

\newcommand*\LyXThinSpace{\,\hspace{0pt}}

\numberwithin{equation}{section}
\numberwithin{figure}{section}
\theoremstyle{plain}
\ifx\thechapter\undefined
	\newtheorem{thm}{\protect\theoremname}
\else
	\newtheorem{thm}{\protect\theoremname}[chapter]
\fi
\theoremstyle{plain}
\newtheorem{lem}[thm]{\protect\lemmaname}
\theoremstyle{plain}
\newtheorem{cor}[thm]{\protect\corollaryname}
\theoremstyle{definition}
\newtheorem{example}[thm]{\protect\examplename}

\makeatother

\usepackage{babel}
\providecommand{\corollaryname}{Corollary}
\providecommand{\examplename}{Example}
\providecommand{\lemmaname}{Lemma}
\providecommand{\theoremname}{Theorem}

\begin{document}
\title{On Almost Orthogonal Series }
\author{Ruiming Zhang}
\email{ruimingzhang@guet.edu.cn}
\address{School of Mathematics and Computing Sciences\\
Guilin University of Electronic Technology\\
Guilin, Guangxi 541004, P. R. China. }
\thanks{This work is supported by the National Natural Science Foundation
of China, grant No. 11771355.}
\begin{abstract}
In this work we prove analogues of Bessel inequality and Riesz-Fisher
theorem in Hilbert spaces with respect to sequences. We apply our
generalized Bessel inequality to the Hilbert spaces associated with
the Normal, Beta, Gamma and certain discrete probability distributions
to show how to generate certain type of inequalities for special functions
systematically. 
\end{abstract}

\maketitle

\section{Introduction }

Let $\left\{ \phi_{n}(x)\right\} _{n\in\mathbb{N}}$ be an orthonormal
sequence of functions in the Hilbert space $L^{2}(a,b)$, then the
infinite matrix $\left(a_{m,n}\right)_{m,n\in\mathbb{N}}$ 
\begin{equation}
a_{n,m}=\int_{a}^{b}\phi_{m}(x)\overline{\phi_{n}(x)}dx=\delta_{m,n},\quad n\in\mathbb{N}\label{eq:1.1}
\end{equation}
defines the identity operator on the Hilbert space $\ell^{2}(\mathbb{N})$,
and for any $f\in L^{2}(a,b)$ the Bessel inequality for the orthonormal
system is given by
\begin{equation}
\sum_{n=1}^{\infty}\left|\int_{a}^{b}f(x)\overline{\phi_{n}(x)}dx\right|^{2}\le\int_{a}^{b}\left|f(x)\right|^{2}dx.\label{eq:1.2}
\end{equation}
In \cite{Bellman} Bellman proved that, under a compact perturbation
to the identity operator on $\ell^{2}(\mathbb{N})$,
\begin{equation}
\left(a_{m,n}\right)_{m,n\in\mathbb{N}}=\left(\delta_{m,n}\right)_{m,n\in\mathbb{N}}+\left(b_{m,n}\right)_{m,n\in\mathbb{N}},\quad\sum_{m,n=1}^{\infty}\left|b_{m,n}\right|^{2}<\infty,\label{eq:1.3}
\end{equation}
certain modified version of Riesz-Fisher theorem and Bessel inequality
are still valid in $L^{2}(a,b)$. 

In this work we show that different versions of modified Riesz-Fisher
theorem and Bessel inequality hold on any Hilbert space under the
condition
\begin{equation}
\sup_{m\in\mathbb{N}}\sum_{n=1}^{\infty}\left|a_{m,n}\right|=\sup_{n\in\mathbb{N}}\sum_{m=1}^{\infty}\left|a_{m,n}\right|<\infty.\label{eq:1.4}
\end{equation}
We apply our generalized Bessel inequality to the Hilbert spaces associated
with the Normal, Beta, Gamma and certain discrete probability distributions
to show how to generate certain type of inequalities for special functions
systematically. It is clear that these are not the only possible choices
for the Hilbert spaces and sequences. For example, one can choose
the Hilbert space to be any $L^{2}$ spaces associated with orthogonal
polynomials on the real line, and the sequence to be functions from
a generalized moment problem, \cite{Andrews,DLMF,Ismail}. 

On the Hilbert space 
\begin{equation}
\ell^{2}(\mathbb{N})=\left\{ \left\{ x_{n}\right\} _{n=1}^{\infty}\bigg|\sum_{n=1}^{\infty}\left|x_{n}\right|^{2}<\infty,\ x_{n}\in\mathbb{C}\right\} \label{eq:1.5}
\end{equation}
we denote

\begin{equation}
\left\langle \left\{ x_{n}\right\} _{n=1}^{\infty},\left\{ y_{n}\right\} _{n=1}^{\infty}\right\rangle _{\ell^{2}}=\sum_{n=1}^{\infty}x_{n}\overline{y_{n}},\label{eq:1.6}
\end{equation}
 and
\begin{equation}
\left\Vert \left\{ x_{n}\right\} _{n=1}^{\infty}\right\Vert _{\ell^{2}}=\sqrt{\left\langle \left\{ x_{n}\right\} _{n=1}^{\infty},\left\{ x_{n}\right\} _{n=1}^{\infty}\right\rangle _{\ell^{2}}}\label{eq:1.7}
\end{equation}
for its inner product and norm. Given any bounded linear operator
$A$ on $\ell^{2}(\mathbb{N})$, its operator norm is given by
\begin{equation}
\left\Vert A\right\Vert _{2}=\sup_{x\neq0,\,x\in\ell^{2}(\mathbb{N})}\frac{\left\Vert Ax\right\Vert _{\ell^{2}}}{\left\Vert x\right\Vert _{\ell^{2}}}=\sup_{\left\Vert x\right\Vert _{\ell^{2}}=1}\left\Vert Ax\right\Vert _{\ell^{2}}.\label{eq:1.8}
\end{equation}
 A complex infinite matrix $\left(a_{m,n}\right)_{m,n\in\mathbb{N}}$
is called symmetric if 
\begin{equation}
\overline{a_{m,n}}=a_{n,m},\quad m,n\in\mathbb{N}.\label{eq:1.9}
\end{equation}

\section{Main Results\label{sec:Main-Results}}
\begin{lem}
\label{lem:1}Let $A=\left(a_{m,n}\right)_{m,n=1}^{\infty}$ be any
symmetric matrix such that
\begin{equation}
C=\sup_{m\in\mathbb{N}}\sum_{n=1}^{\infty}\left|a_{m,n}\right|=\sup_{n\in\mathbb{N}}\sum_{m=1}^{\infty}\left|a_{m,n}\right|<\infty.\label{eq:2.1}
\end{equation}
Then $A$ defines a bounded self-adjoint operator on $\ell^{2}(\mathbb{N})$
satisfies 
\begin{equation}
\left\Vert A\right\Vert _{2}=\sup_{\left\Vert x\right\Vert _{\ell^{2}}=1}\left\Vert Ax\right\Vert _{\ell^{2}}\le C.\label{eq:2.2}
\end{equation}
Furthermore, if it satisfies any of the following three additional
conditions:

\begin{equation}
\sum_{m,n=1}^{\infty}\left|a_{n,m}\right|^{2}<\infty,\label{eq:2.3}
\end{equation}
\begin{equation}
\lim_{N\to\infty}\sup_{m\ge N+1}\left(\sum_{n=1}^{\infty}\left|a_{m,n}\right|\right)=0,\label{eq:2.4}
\end{equation}
 or
\begin{equation}
\lim_{N\to\infty}\sup_{n\ge1}\left(\sum_{m=N+1}^{\infty}\left|a_{n,m}\right|\right)=0,\label{eq:2.5}
\end{equation}
then $A$ is a compact self-adjoint operator on $\ell^{2}(\mathbb{N})$. 
\end{lem}

\begin{proof}
For any $N\in\mathbb{N}$ the operator norms of $A^{(N)}=\left(a_{m,n}\right)_{m,n=1}^{N}$
on $\ell^{1}(\mathbb{C}^{N}),\ell^{2}(\mathbb{C}^{N}),\ell^{\infty}(\mathbb{C}^{N})$
satisfy \cite{HornJohnson} 
\[
\left\Vert A^{(N)}\right\Vert _{1}=\max_{1\le n\le N}\sum_{m=1}^{N}\left|a_{m,n}\right|\le C,\quad\left\Vert A^{(N)}\right\Vert _{\infty}=\max_{1\le m\le N}\sum_{n=1}^{N}\left|a_{m,n}\right|\le C
\]
and
\[
\left\Vert A^{(N)}\right\Vert _{2}\le\sqrt{\left\Vert A^{(N)}\right\Vert _{1}\cdot\left\Vert A^{(N)}\right\Vert _{\infty}}\le C.
\]
Then for any $N\in\mathbb{N}$ and $x=\left\{ x_{n}\right\} _{n=1}^{\infty},\,y=\left\{ y_{n}\right\} _{n=1}^{\infty}\in\ell^{2}(\mathbb{N})$,
\[
\sum_{m,n=1}^{N}\left|a_{m,n}x_{m}\overline{y_{n}}\right|\le C\left\Vert x\right\Vert _{\ell^{2}}\cdot\left\Vert y\right\Vert _{\ell^{2}}<\infty.
\]
Let $N\to\infty$ in the above inequality to obtain
\[
\sum_{m,n=1}^{\infty}\left|a_{m,n}x_{m}\overline{y_{n}}\right|=\sum_{n,m=1}^{\infty}\left|a_{m,n}x_{m}\overline{y_{n}}\right|\le C\left\Vert x\right\Vert _{\ell^{2}}\cdot\left\Vert y\right\Vert _{\ell^{2}}<\infty.
\]
Then by Fubini's theorem,
\[
<x,Ay>_{\ell^{2}}=<Ax,y>_{\ell^{2}}
\]
 and
\[
\left|<x,Ay>_{\ell^{2}}\right|=\left|<Ax,y>_{\ell^{2}}\right|\le C\left\Vert x\right\Vert _{\ell^{2}}\cdot\left\Vert y\right\Vert _{\ell^{2}}
\]
hold for any $x,y\in\ell^{2}(\mathbb{N})$. Thus $A$ is self-adjoint
with $\left\Vert A\right\Vert _{2}\le C$, \cite{AkhiezerGlazman}. 

For each $N\in\mathbb{N}$, let
\[
A^{(N)}=\left(a_{m,n}^{(N)}\right)_{m,n=1}^{\infty},\quad a_{m,n}^{(N)}=\begin{cases}
a_{m,n}, & 1\le m\le N,\\
0, & m\ge N+1.
\end{cases}
\]
 For $x\in\ell^{2}(\mathbb{N}_{0})$ the $m$-th component of $A^{(N)}x$
is 
\[
\left(A^{(N)}x\right)_{m}=\begin{cases}
\sum_{n=1}^{\infty}a_{m,n}x_{n}, & 1\le m\le N,\\
0, & m\ge N+1.
\end{cases}
\]
Then the range of $A^{(N)}$ is $N$ dimensional. Therefore $A^{(N)}$
is a compact operator on $\ell^{2}(\mathbb{N})$, \cite{AkhiezerGlazman}. 

If \eqref{eq:2.3} is satisfied, then 
\[
\lim_{N\to+\infty}\sum_{m=N+1}^{\infty}\sum_{n=1}^{\infty}\left|a_{n,m}\right|^{2}=0
\]
and
\[
\begin{aligned} & \left\Vert A-A^{(N)}\right\Vert _{2}^{2}=\sup_{\left\Vert x\right\Vert _{\ell^{2}}=1}\left\Vert \left(A-A^{(N)}\right)x\right\Vert _{\ell^{2}}^{2}=\sup_{\left\Vert x\right\Vert _{\ell^{2}}=1}\sum_{m=N+1}^{\infty}\left|\sum_{n=1}^{\infty}a_{m,n}x_{n}\right|^{2}\\
 & \le\sup_{\left\Vert x\right\Vert _{\ell^{2}}=1}\left(\sum_{m=N+1}^{\infty}\sum_{n=1}^{\infty}\left|a_{n,m}\right|^{2}\right)\cdot\left(\sum_{n=1}^{\infty}\left|x_{n}\right|^{2}\right)\le\left(\sum_{m=N+1}^{\infty}\sum_{n=1}^{\infty}\left|a_{n,m}\right|^{2}\right).
\end{aligned}
\]
Hence, $A=\lim_{N\to+\infty}A^{(N)}$ in operator norm on $\ell^{2}(\mathbb{N})$.
Since compact operators form a closed ideal in the norm limit, then
$A$ is compact, \cite{AkhiezerGlazman}. 

Similarly, observe that
\[
\begin{aligned} & \left\Vert A-A^{(N)}\right\Vert _{2}^{2}=\sup_{\left\Vert x\right\Vert _{\ell^{2}}=1}\left\Vert \left(A-A^{(N)}\right)x\right\Vert _{\ell^{2}}^{2}=\sup_{\left\Vert x\right\Vert _{\ell^{2}}=1}\sum_{m=N+1}^{\infty}\left|\sum_{n=1}^{\infty}a_{m,n}x_{n}\right|^{2}\\
 & \le\sup_{\left\Vert x\right\Vert _{\ell^{2}}=1}\sum_{m=N+1}^{\infty}\left(\sum_{n=1}^{\infty}\sqrt{\left|a_{m,n}\right|}\cdot\left(\sqrt{\left|a_{m,n}\right|}\left|x_{n}\right|\right)\right)^{2}\\
 & \le\sup_{\left\Vert x\right\Vert _{\ell^{2}}=1}\sum_{m=N+1}^{\infty}\left(\sum_{n=1}^{\infty}\left|a_{m,n}\right|\right)\cdot\left(\sum_{n=1}^{\infty}\left|a_{m,n}\right|\cdot\left|x_{n}\right|^{2}\right)\\
 & \le\sup_{m\ge N+1}\left(\sum_{n=1}^{\infty}\left|a_{m,n}\right|\right)\cdot\sup_{\left\Vert x\right\Vert _{\ell^{2}}=1}\left(\sum_{m=N+1}^{\infty}\sum_{n=1}^{\infty}\left|a_{m,n}\right|\cdot\left|x_{n}\right|^{2}\right)\\
 & \le\sup_{m\ge N+1}\left(\sum_{n=1}^{\infty}\left|a_{m,n}\right|\right)\cdot\sup_{\left\Vert x\right\Vert _{\ell^{2}}=1}\left(\sum_{n=1}^{\infty}\left(\sum_{m=N+1}^{\infty}\left|a_{n,m}\right|\right)\cdot\left|x_{n}\right|^{2}\right)\\
 & \le\sup_{m\ge N+1}\left(\sum_{n=1}^{\infty}\left|a_{m,n}\right|\right)\cdot\sup_{n\ge1}\left(\sum_{m=N+1}^{\infty}\left|a_{n,m}\right|\right)\cdot\sup_{\left\Vert x\right\Vert _{\ell^{2}}=1}\left(\sum_{n=1}^{\infty}\left|x_{n}\right|^{2}\right)\\
 & =\sup_{m\ge N+1}\left(\sum_{n=1}^{\infty}\left|a_{m,n}\right|\right)\cdot\sup_{n\ge1}\left(\sum_{m=N+1}^{\infty}\left|a_{n,m}\right|\right).
\end{aligned}
\]
If \eqref{eq:2.4} or \eqref{eq:2.5} is satisfied, then we have
\[
\left\Vert A-A^{(N)}\right\Vert _{2}^{2}\le C\sup_{m\ge N+1}\left(\sum_{n=1}^{\infty}\left|a_{m,n}\right|\right)
\]
or
\[
\left\Vert A-A^{(N)}\right\Vert _{2}^{2}\le C\sup_{n\ge1}\left(\sum_{m=N+1}^{\infty}\left|a_{n,m}\right|\right).
\]
Therefore, under \eqref{eq:2.4} or \eqref{eq:2.5}, $A$ is still
the norm limit of compact operators $A^{(N)}$. Therefore, $A$ must
be compact. 
\end{proof}
Given a Hilbert space $\mathcal{H}$ with inner product $(\cdot,\cdot)$
and norm $\left\Vert \cdot\right\Vert $, let $\left\{ \phi_{n}\right\} _{n=1}^{\infty}$
be a sequence in $\mathcal{H}$. We define $A=\left(a_{m,n}\right)_{m,n=1}^{\infty}$
by

\begin{equation}
a_{m,n}=(\phi_{m},\phi_{n})=\overline{(\phi_{n},\phi_{m})}=\overline{a_{n,m}},\quad m,n\in\mathbb{N}.\label{eq:2.6}
\end{equation}
If \eqref{eq:2.1} is satisfied, then by Lemma \ref{lem:1} $A$ is
a bounded self-adjoint operator. Furthermore, it is positive semidefinite,
since for any $x\in\ell^{2}(\mathbb{N})$,

\begin{equation}
\begin{aligned} & <x,Ax>_{\ell^{2}}=\lim_{N\to\infty}\sum_{m=1}^{N}x_{m}\overline{\sum_{n=1}^{N}a_{m,n}x_{n}}=\lim_{N\to\infty}\sum_{m=1}^{N}\sum_{n=1}^{N}(\phi_{n},\phi_{m})x_{m}\overline{x_{n}}\\
 & =\lim_{N\to\infty}\left(\sum_{n=1}^{N}\overline{x_{n}}\phi_{n},\sum_{m=1}^{N}\overline{x_{m}}\phi_{m}\right)=\lim_{N\to\infty}\left\Vert \sum_{n=1}^{N}\overline{x_{n}}\phi_{n}\right\Vert ^{2}\ge0.
\end{aligned}
\label{eq:2.7}
\end{equation}

The following is a generalization of the Bessel inequality.
\begin{thm}
\label{thm:2}Let $\left\{ \phi_{n}\right\} _{n=1}^{\infty}$ be a
sequence in Hilbert space $\mathcal{H}$ and $A=\left(a_{m,n}\right)_{m,n=1}^{\infty}$
defined as in \eqref{eq:2.6} such that \eqref{eq:2.1} holds. If
$f\in\mathcal{H}$ then 
\begin{equation}
\sum_{n=1}^{\infty}\left|(f,\phi_{n})\right|^{2}\le C\left\Vert f\right\Vert ^{2}.\label{eq:2.8}
\end{equation}
 
\end{thm}

\begin{proof}
For any $N\in\mathbb{N}$,
\[
\begin{aligned} & \sum_{n=1}^{N}\left|(f,\phi_{n})\right|^{2}=\left(f,\sum_{n=1}^{N}(f,\phi_{n})\phi_{n}\right)\le\left\Vert f\right\Vert \cdot\left\Vert \sum_{n=1}^{N}(f,\phi_{n})\phi_{n}\right\Vert \\
 & =\left\Vert f\right\Vert \cdot\sqrt{\left\Vert \sum_{n=1}^{N}(f,\phi_{n})\phi_{n}\right\Vert ^{2}}=\left\Vert f\right\Vert \cdot\sqrt{\left(\sum_{m=1}^{N}(f,\phi_{m})\phi_{m},\sum_{n=1}^{N}(f,\phi_{n})\phi_{n}\right)}\\
 & =\left\Vert f\right\Vert \cdot\sqrt{\sum_{m,n=1}^{N}a_{m,n}(f,\phi_{m})\overline{(f,\phi_{n})}}.
\end{aligned}
\]
 Notice that
\[
0\le\sum_{m,n=1}^{N}a_{m,n}(f,\phi_{m})\overline{(f,\phi_{n})}\le C\sum_{n=1}^{N}\left|(f,\phi_{n})\right|^{2},
\]
hence, 
\[
\sqrt{\sum_{n=1}^{N}\left|(f,\phi_{n})\right|^{2}}\le\sqrt{C}\cdot\left\Vert f\right\Vert ,
\]
which gives
\[
\sum_{n=1}^{N}\left|(f,\phi_{n})\right|^{2}\le C\left\Vert f\right\Vert ^{2}
\]
for any $N\in\mathbb{N}$. \eqref{eq:2.8} is obtained by letting
$N\to\infty$.
\end{proof}
The following is an analogue of the Riesz-Fisher theorem.
\begin{thm}
\label{thm:3}Under the assumptions of Theorem \ref{thm:2}. For any
$\left\{ x_{n}\right\} _{n=1}^{\infty}\in\ell^{2}(\mathbb{N})$, there
exists a $f\in\mathcal{H}$ such that
\begin{equation}
\sum_{n=1}^{\infty}\left|x_{n}-(f,\phi_{n})\right|^{2}\le C^{2}\sum_{n=1}^{\infty}\left|x_{n}\right|^{2}.\label{eq:2.9}
\end{equation}
 Consequently, 
\begin{equation}
\lim_{n\to+\infty}\left(x_{n}-(f,\phi_{n})\right)=0.\label{eq:2.10}
\end{equation}
\end{thm}

\begin{proof}
For any $\left\{ x_{n}\right\} _{n=1}^{\infty}\in\ell^{2}(\mathbb{N})$,
let
\[
s_{n}=\sum_{k=1}^{n}x_{k}\phi_{k}\in\mathcal{H},
\]
 then for $m>n\ge1$,
\[
\begin{aligned} & \left\Vert s_{m}-s_{n}\right\Vert ^{2}=\left(\sum_{j=n+1}^{m}x_{j}\phi_{j}(x),\sum_{k=n+1}^{m}x_{k}\phi_{k}(x)\right)=\sum_{j,k=n+1}^{m}a_{j,k}x_{j}\overline{x_{k}}\\
 & \le\left(\sup_{m\ge j\ge n+1}\sum_{k=n+1}^{m}\left|a_{j,k}\right|\right)\cdot\sum_{j=n+1}^{m}\left|x_{j}\right|^{2}\le C\sum_{j=n+1}^{m}\left|x_{j}\right|^{2}
\end{aligned}
\]
Hence $\left\{ s_{n}\right\} _{n=1}^{\infty}$ is a Cauchy sequence
in the Hilbert space $\mathcal{H}$. Therefore, there exists $f\in\mathcal{H}$
such that
\[
\lim_{n\to\infty}\left\Vert s_{n}-f\right\Vert =0.
\]
 For any $n<N$, 
\[
\begin{aligned} & \left|x_{n}-(f,\phi_{n})\right|=\left|x_{n}-(s_{N},\phi_{n})-(f-s_{N},\phi_{n})\right|=\left|x_{n}-\sum_{m=1}^{N}x_{m}(\phi_{m},\phi_{n})-(f-s_{N},\phi_{n})\right|\\
 & \le\left|x_{n}-\sum_{m=1}^{N}x_{m}a_{m,n}\right|+\left|(f-s_{N},\phi_{n})\right|\le\sum_{m=1,m\neq n}^{N}\left|x_{m}\right|\cdot\left|a_{m,n}\right|+\left\Vert f-s_{N}\right\Vert \cdot\left\Vert \phi_{n}\right\Vert \\
 & \le\sum_{m=1}^{N}\left|x_{m}\right|\cdot\left|a_{m,n}\right|+\left\Vert f-s_{N}\right\Vert .
\end{aligned}
\]
Hence,
\[
\left|x_{n}-(f,\phi_{n})\right|\le\limsup_{N\to\infty}\left(\sum_{m=1}^{N}\left|x_{m}\right|\cdot\left|a_{m,n}\right|+\left\Vert f-s_{N}\right\Vert \right)=\sum_{m=1}^{\infty}\left|x_{m}\right|\cdot\left|a_{m,n}\right|
\]
Therefore,
\[
\begin{aligned} & \sum_{n=1}^{\infty}\left|x_{n}-(f,\phi_{n})\right|^{2}\le\sum_{n=1}^{\infty}\left(\sum_{m=1}^{\infty}\left|x_{m}\right|\cdot\left|a_{m,n}\right|\right)^{2}=\sum_{n=1}^{\infty}\left(\sum_{m=1}^{\infty}\left|x_{m}\right|\cdot\left|a_{n,m}\right|\right)^{2}\\
 & =\left\langle A_{1}z,A_{1}z\right\rangle _{\ell^{2}}\le\left\Vert A_{1}\right\Vert _{2}^{2}\cdot\left\Vert z\right\Vert _{\ell^{2}}^{2}\le C^{2}\left\Vert x\right\Vert _{\ell^{2}}^{2},
\end{aligned}
\]
where $A_{1}=\left(\left|a_{m,n}\right|\right)_{m,n=1}^{\infty}$
and $z=\left\{ \left|x_{m}\right|\right\} _{m=1}^{\infty}$. 
\end{proof}
Let $\psi(x)$ be an nondecreasing function on $\mathbb{R}$ such
that \cite{Andrews,DLMF,Ismail}
\begin{equation}
\int_{-\infty}^{\infty}d\psi(x)=1.\label{eq:2.11}
\end{equation}
For $\mathcal{H}=L^{2}(\mathbb{R},d\psi(x))$ and a sequence of almost
periodic functions we have the following corollary of Theorems \ref{thm:2}
and \ref{thm:3}. 
\begin{cor}
\label{cor:4} Given a sequence of real numbers $\left\{ \lambda_{n}\right\} _{n\in\mathbb{N}}$
let
\begin{equation}
\phi_{n}(x)=e^{i\lambda_{n}x},\quad n\in\mathbb{N}\label{eq:2.12}
\end{equation}
and 
\begin{equation}
a_{m,n}=\int_{-\infty}^{\infty}\phi_{m}(x)\overline{\phi_{n}(x)}d\psi(x)=\int_{-\infty}^{\infty}e^{i(\lambda_{m}-\lambda_{n})x}d\psi(x).\label{eq:2.13}
\end{equation}

If 
\begin{equation}
C_{4}=\sup_{m\in\mathbb{N}}\sum_{n=1}^{\infty}\left|\int_{-\infty}^{\infty}e^{i(\lambda_{m}-\lambda_{n})x}d\psi(x)\right|=\sup_{n\in\mathbb{N}}\sum_{m=1}^{\infty}\left|\int_{-\infty}^{\infty}e^{i(\lambda_{m}-\lambda_{n})x}d\psi(x)\right|<\infty,\label{eq:2.14}
\end{equation}
 then for any $f\in L^{2}(\mathbb{R},d\psi(x))$ we have the generalized
Bessel inequality,
\begin{equation}
\sum_{n=1}^{\infty}\left|\int_{\mathbb{R}}f(x)e^{-i\lambda_{n}x}d\psi(x)\right|^{2}\le C_{4}\int_{\mathbb{R}}\left|f(x)\right|^{2}d\psi(x).\label{eq:2.15}
\end{equation}
Furthermore, for any $\left\{ x_{n}\right\} _{n=1}^{\infty}\in\ell^{2}(\mathbb{N})$,
there exists a function $F\in L^{2}\left(\mathbb{R},d\psi(x)\right)$
such that
\begin{equation}
\sum_{n=1}^{\infty}\left|x_{n}-\int_{\mathbb{R}}F(x)e^{-i\lambda_{n}x}d\psi(x)\right|^{2}<\infty.\label{eq:2.16}
\end{equation}
\end{cor}

The following is another direct consequences Theorems \ref{thm:2}
and \ref{thm:3} when the support of $d\psi$ is a subset of $\mathbb{N}$:
\begin{cor}
\label{cor:5} Let $\left\{ \psi(k)\right\} _{k=1}^{\infty}$ be a
sequence of nonnegative numbers such that $\sum_{k\in\mathbb{N}}\psi(k)=1$.
Then for any bounded function $f(x)$ the Stieltjes integral is given
by
\begin{equation}
\int_{\mathbb{R}}f(x)d\psi(x)=\sum_{k=1}^{\infty}\psi(k)f(k).\label{eq:2.17}
\end{equation}
Let $\left\{ \phi_{n}(x)\right\} _{n}$ be a sequence of functions
in $L^{2}\left(\mathbb{R},d\psi(x)\right)$, we define,
\begin{equation}
a_{m,n}=\sum_{k=1}^{\infty}\psi(k)\phi_{m}(k)\overline{\phi_{n}(k)},\quad m,n\in\mathbb{N}.\label{eq:2.18}
\end{equation}
If it satisfies
\begin{equation}
C_{5}=\sup_{m\in\mathbb{N}}\sum_{n=1}^{\infty}\left|a_{m,n}\right|=\sup_{n\in\mathbb{N}}\sum_{m=1}^{\infty}\left|a_{m,n}\right|<\infty,\label{eq:2.19}
\end{equation}
then for any $f\in L^{2}(\mathbb{R},d\psi(x))$, i.e.
\begin{equation}
\sum_{k=1}^{\infty}\psi(k)\left|f(k)\right|^{2}<\infty,\label{eq:2.20}
\end{equation}
we have
\begin{equation}
\sum_{n=1}^{\infty}\left|f_{n}\right|^{2}\le\left(\sup_{m\in\mathbb{N}}\sum_{n=1}^{\infty}\left|a_{m,n}\right|\right)\sum_{k=1}^{\infty}\psi(k)\left|f(k)\right|^{2},\label{eq:2.21}
\end{equation}
 where
\begin{equation}
f_{n}=\sum_{k=1}^{\infty}\psi(k)f(k)\overline{\phi_{n}(k)}.\label{eq:2.22}
\end{equation}
\end{cor}

\section{\label{sec:Examples}Applications}

In this section we apply Corollary \ref{cor:4} to derive inequalities
for the special functions associated with the Gaussian, Gamma and
Beta distributions. We also provide discrete examples related to Dirichlet
series. Our purpose here is to show that generalized Bessel inequalities
can be used to generate certain type special function inequalities. 

\subsection{Gaussian distribution}

Since \cite{Andrews,DLMF,Ismail}
\begin{equation}
\int_{-\infty}^{\infty}\frac{e^{-x^{2}/2+i\lambda x}dx}{\sqrt{2\pi}}=e^{-\lambda^{2}/2},\label{eq:3.1}
\end{equation}
then for any sequence of positive numbers $\left\{ \lambda_{n}\right\} _{n\in\mathbb{N}}$
let $\phi_{n}(x)=e^{i\lambda_{n}x},\ x\in\mathbb{R},\ n\in\mathbb{N}$
and 
\begin{equation}
a_{m,n}=\int_{-\infty}^{\infty}\frac{e^{-x^{2}/2+i(\lambda_{m}-\lambda_{n})x}dx}{\sqrt{2\pi}}=e^{-(\lambda_{m}-\lambda_{n})^{2}/2}.\label{eq:3.2}
\end{equation}
Assuming there exists a positive number $\alpha>\sqrt{2}$ such that
for positive integers $m>n\ge1$
\begin{equation}
\lambda_{m}-\lambda_{n}\ge\alpha\log^{1/2}\left(1+\left|m-n\right|\right),\label{eq:3.3}
\end{equation}
 then for any $m\in\mathbb{N}$,
\[
\begin{aligned} & \sum_{n=1}^{\infty}\left|a_{m,n}\right|\le\sum_{n=1}^{\infty}e^{-\alpha^{2}/2\log\left(1+\left|m-n\right|\right)}=\sum_{n=1}^{\infty}\frac{1}{\left(1+\left|m-n\right|\right)^{\alpha^{2}/2}}\\
 & \le\sum_{n=-\infty}^{\infty}\frac{1}{\left(1+\left|n\right|\right)^{\alpha^{2}/2}}<\infty,
\end{aligned}
\]
then
\begin{equation}
C_{1}=\sup_{m\in\mathbb{N}}\sum_{n=1}^{\infty}e^{-(\lambda_{m}-\lambda_{n})^{2}/2}\le\sum_{n=-\infty}^{\infty}\frac{1}{\left(1+\left|n\right|\right)^{\alpha^{2}/2}}<\infty.\label{eq:3.4}
\end{equation}
In particular, if $\lambda_{n}=n$ and $\phi_{n}(x)=e^{inx}$ then
\begin{equation}
a_{m,n}=e^{-(m-n)^{2}/2},\quad b_{m,n}=a_{m,n}-\delta_{m,n}\in\mathbb{R},\label{eq:3.5}
\end{equation}
then 
\begin{equation}
\sum_{m,n=1}^{\infty}b_{m,n}^{2}\ge\sum_{n=1}^{\infty}b_{n+1,n}^{2}=\sum_{n=1}^{\infty}e^{-1}=\infty,\label{eq:3.6}
\end{equation}
while
\begin{equation}
\begin{aligned} & C_{1}=\sup_{m\in\mathbb{N}}\sum_{n=1}^{\infty}\left|a_{m,n}\right|=\sup_{m\in\mathbb{N}}\sum_{n=1}^{\infty}e^{-(m-n)^{2}/2}\\
 & \le\sup_{m\in\mathbb{N}}\sum_{n=-\infty}^{\infty}e^{-(m-n)^{2}/2}=\sum_{n=-\infty}^{\infty}e^{-n^{2}/2}<\infty.
\end{aligned}
\label{eq:3.7}
\end{equation}
 The above shows that \eqref{eq:1.4} type condition is satisfied
while \eqref{eq:1.3} is not met.

Recall that for $a,n,a_{1},\dots,a_{m}\in\mathbb{C}$ the $q$-shifted
factorials are defined by \cite{Andrews,Ismail,IsmailZhang1}
\begin{equation}
(a;q)_{\infty}=\prod_{k=0}^{\infty}(1-aq^{k}),\,(a;q)_{n}=\frac{(a;q)_{\infty}}{(aq^{n};q)_{\infty}}\label{eq:3.8}
\end{equation}
and 
\begin{equation}
(a_{1},a_{2},\cdots,a_{m};q)_{n}=\prod_{j=1}^{m}(a_{j};q)_{n}.\label{eq:3.9}
\end{equation}
For $0<q<1$ and $\beta>0$, let $\lambda=\mu\sqrt{\log q^{-2\beta}}$
in \eqref{eq:3.1} to get
\begin{equation}
\int_{-\infty}^{\infty}\psi(x\vert\beta)e^{i\mu x}dx=q^{\beta\mu^{2}},\label{eq:3.10}
\end{equation}
where 
\begin{equation}
\psi(x\vert\beta)=\frac{\exp\left(\frac{x^{2}}{\log q^{4\beta}}\right)}{\sqrt{2\pi\log q^{-2\beta}}}.\label{eq:3.11}
\end{equation}
 Then for any sequence of positive numbers $\left\{ \mu_{n}\right\} _{n\in\mathbb{N}}$
and the sequence of functions $\left\{ e^{i\mu_{n}x}\right\} _{n\in\mathbb{N}}$
such that $\alpha>\sqrt{2}$ and for $m>n\ge1$,
\begin{equation}
\mu_{m}-\mu_{n}\ge\frac{\alpha}{\sqrt{\log q^{-2\beta}}}\log^{1/2}\left(1+\left|m-n\right|\right),\label{eq:3.12}
\end{equation}
 we have
\begin{equation}
a_{m,n}=\int_{-\infty}^{\infty}e^{i(\mu_{m}-\mu_{n})x}\psi(x\vert\beta)dx=q^{\beta(\mu_{m}-\mu_{n})^{2}}\label{eq:3.13}
\end{equation}
and
\begin{equation}
C_{1}(\beta)=\sup_{m\in\mathbb{N}}\sum_{n=1}^{\infty}q^{\beta(\mu_{m}-\mu_{n})^{2}}<\infty.\label{eq:3.14}
\end{equation}

\begin{example}
For $\beta=\frac{1}{2},\ \alpha>\sqrt{2},\ \left|z\right|<1$ let
\[
f(x)=\frac{1}{\left(ze^{ix};q\right)_{\infty}}
\]
and a sequence of positive numbers $\left\{ \mu_{n}\right\} $
\[
\mu_{m}-\mu_{n}\ge\frac{\alpha}{\sqrt{\log q^{-1}}}\log^{1/2}\left(1+\left|m-n\right|\right).
\]
Then by \cite[pp21, (5.12)]{IsmailZhang1},
\[
f_{n}=\int_{-\infty}^{\infty}f(x)e^{-i\mu_{n}x}d\psi\left(x\vert\frac{1}{2}\right)=q^{\mu_{n}^{2}/2}\left(-zq^{1/2-\mu_{n}};q\right)_{\infty}.
\]
By the $q$-Binomial theorem,
\[
\begin{aligned} & \int_{-\infty}^{\infty}\left|f(x)\right|^{2}\psi\left(x\vert\frac{1}{2}\right)=\frac{1}{\sqrt{2\pi\log q^{-1}}}\int_{-\infty}^{\infty}\frac{\exp\left(\frac{x^{2}}{\log q^{2}}\right)dx}{\left(ze^{ix},\overline{z}e^{-ix};q\right)_{\infty}}\\
 & =\sum_{n=0}^{\infty}\frac{z^{n}}{\sqrt{2\pi\log q^{-1}}(q;q)_{n}}\int_{-\infty}^{\infty}\frac{\exp\left(\frac{x^{2}}{\log q^{2}}+inx\right)dx}{\left(\overline{z}e^{-ix};q\right)_{\infty}}\\
 & =\sum_{n=0}^{\infty}\frac{z^{n}q^{n^{2}/2}(-\overline{z}q^{n+1/2};q)_{\infty}}{(q;q)_{n}}=(-\overline{z}q^{1/2};q)_{\infty}\sum_{n=0}^{\infty}\frac{z^{n}q^{n^{2}/2}}{(q,-\overline{z}q^{1/2};q)_{n}}\\
 & =(-zq^{1/2};q)_{\infty}\sum_{n=0}^{\infty}\frac{\overline{z}^{n}q^{n^{2}/2}}{(q,-zq^{1/2};q)_{n}},
\end{aligned}
\]
which also shows that for $\left|z\right|<q^{1/2}$,
\begin{equation}
(\overline{z};q)_{\infty}\sum_{n=0}^{\infty}\frac{(-z)^{n}q^{\binom{n}{2}}}{(q,\overline{z};q)_{n}}=(z;q)_{\infty}\sum_{n=0}^{\infty}\frac{(-\overline{z})^{n}q^{\binom{n}{2}}}{(q,z;q)_{n}}>0.\label{eq:3.15}
\end{equation}
 By the generalized Bessel inequality \eqref{eq:2.15} we get
\begin{equation}
\sum_{n=1}^{\infty}q^{\mu_{n}^{2}}\left|\left(-zq^{1/2-\mu_{n}};q\right)_{\infty}\right|^{2}\le\sup_{m\in\mathbb{N}}\sum_{n=1}^{\infty}q^{(\mu_{m}-\mu_{n})^{2}/2}\sum_{n=0}^{\infty}\frac{(-zq^{1/2};q)_{\infty}\overline{z}^{n}q^{n^{2}/2}}{(q,-zq^{1/2};q)_{n}},\label{eq:3.16}
\end{equation}
where $\left|z\right|<1$.
\end{example}

\begin{example}
Let 
\[
\beta=\frac{1}{2},\ \alpha>\sqrt{2},\ z\in\mathbb{C},\quad f(x)=\left(zq^{1/2}e^{ix};q\right)_{\infty}
\]
and a sequence of positive numbers $\left\{ \mu_{n}\right\} _{n\in\mathbb{N}}$
such that 
\[
\mu_{m}-\mu_{n}\ge\frac{\alpha}{\sqrt{\log q^{-1}}}\log^{1/2}\left(1+\left|m-n\right|\right).
\]
Then by \cite[(5.34)]{IsmailZhang1} 
\[
f_{n}=\int_{-\infty}^{\infty}f(x)e^{-i\mu_{n}x}\psi\left(x\vert\frac{1}{2}\right)dx=q^{\mu_{n}^{2}}A_{q}\left(q^{-\mu_{n}}z\right),
\]
where the Ramanujan function is defined by \cite{Ismail,IsmailZhang1}
\[
A_{q}(z)=\sum_{n=0}^{\infty}\frac{q^{n^{2}}(-z)^{n}}{(q;q)_{n}}.
\]
 By the $q$-Binomial theorem and \cite[(5.34)]{IsmailZhang1} ,
\[
\begin{aligned} & \int_{-\infty}^{\infty}\left|f(x)\right|^{2}\psi\left(x\vert\frac{1}{2}\right)dx=\int_{-\infty}^{\infty}\frac{\left(zq^{1/2}e^{ix},\overline{z}q^{1/2}e^{-ix};q\right)_{\infty}\exp\left(\frac{x^{2}}{\log q^{2}}\right)}{\sqrt{2\pi\log q^{-1}}}dx\\
 & =\sum_{n=0}^{\infty}\frac{(-z)^{n}q^{n^{2}}}{(q;q)_{n}}A_{q}\left(q^{-n}\overline{z}\right)=\sum_{n=0}^{\infty}\frac{(-\overline{z})^{n}q^{n^{2}}}{(q;q)_{n}}A_{q}\left(q^{-n}z\right),
\end{aligned}
\]
which also gives 
\begin{equation}
\sum_{n=0}^{\infty}\frac{(-z)^{n}q^{n^{2}}}{(q;q)_{n}}A_{q}\left(q^{-n}\overline{z}\right)=\sum_{n=0}^{\infty}\frac{(-\overline{z})^{n}q^{n^{2}}}{(q;q)_{n}}A_{q}\left(q^{-n}z\right)>0.\label{eq:3.17}
\end{equation}
Then for $z\in\mathbb{C}$ by \eqref{eq:2.15} we have
\begin{equation}
\begin{aligned} & \sum_{n=1}^{\infty}q^{2\mu_{n}^{2}}\left|A_{q}\left(q^{-\mu_{n}}z\right)\right|^{2}\\
 & \le\sup_{m\in\mathbb{N}}\sum_{n=1}^{\infty}q^{(\mu_{m}-\mu_{n})^{2}/2}\sum_{n=0}^{\infty}\frac{(-z)^{n}q^{n^{2}}}{(q;q)_{n}}A_{q}\left(q^{-n}\overline{z}\right).
\end{aligned}
\label{eq:3.18}
\end{equation}
\end{example}

\begin{example}
Let 
\[
\beta=1,\ \alpha>\sqrt{2},\ \left|z\right|<1,\quad f(x)=\frac{1}{\left(-ze^{ix};q\right)_{\infty}}
\]
 and a sequence of positive numbers $\left\{ \mu_{n}\right\} _{n\in\mathbb{N}}$
such that 
\[
\mu_{m}-\mu_{n}\ge\frac{\alpha}{\sqrt{\log q^{-2}}}\log^{1/2}\left(1+\left|m-n\right|\right).
\]
Then by \cite[(5.36)]{IsmailZhang1},
\[
f_{n}=\int_{-\infty}^{\infty}f(x)e^{-i\mu_{n}x}\psi\left(x\vert1\right)dx=q^{\mu_{n}^{2}}A_{q}\left(q^{-2\mu_{n}}z\right).
\]
 By the $q$-Binomial theorem and \cite[(5.36)]{IsmailZhang1}, 
\[
\begin{aligned} & \int_{-\infty}^{\infty}\left|f(x)\right|^{2}\psi\left(x\vert1\right)dx=\int_{-\infty}^{\infty}\frac{\exp\left(\frac{x^{2}}{\log q^{4}}\right)dx}{\left(-ze^{ix},-\overline{z}e^{-ix};q\right)_{\infty}\sqrt{\pi\log q^{-4}}}\\
 & =\sum_{n=0}^{\infty}\frac{(-z)^{n}}{(q;q)_{n}}\int_{-\infty}^{\infty}\frac{\exp\left(\frac{x^{2}}{\log q^{4}}+inx\right)dx}{\left(-\overline{z}e^{-ix};q\right)_{\infty}\sqrt{\pi\log q^{-4}}}=\sum_{n=0}^{\infty}\frac{(-z)^{n}}{(q;q)_{n}}q^{n^{2}}A_{q}\left(q^{-2n}\overline{z}\right),
\end{aligned}
\]
which also gives
\begin{equation}
\sum_{n=0}^{\infty}\frac{(-z)^{n}}{(q;q)_{n}}q^{n^{2}}A_{q}\left(q^{-2n}\overline{z}\right)=\sum_{n=0}^{\infty}\frac{(-\overline{z})^{n}}{(q;q)_{n}}q^{n^{2}}A_{q}\left(q^{-2n}z\right)>0,\label{eq:3.19}
\end{equation}
 where $\left|z\right|<1$. 

Then by \eqref{eq:2.15},
\begin{equation}
\begin{aligned} & \sum_{n=1}^{\infty}q^{2\mu_{n}^{2}}\left|A_{q}\left(q^{-2\mu_{n}}z\right)\right|^{2}\\
 & \le\sup_{m\in\mathbb{N}}\sum_{n=1}^{\infty}q^{(\mu_{m}-\mu_{n})^{2}}\sum_{n=0}^{\infty}\frac{(-z)^{n}}{(q;q)_{n}}q^{n^{2}}A_{q}\left(q^{-2n}\overline{z}\right).
\end{aligned}
\label{eq:3.20}
\end{equation}
\end{example}

\begin{example}
For $x=\cos\theta$, $\theta\in[0,\pi]$ and $\left|t\right|<1$ let
\[
\beta=\frac{1}{4},\ \alpha>\sqrt{2},\quad f(y)=\frac{1}{\left(te^{i\left(y+\theta\right)},te^{i\left(y-\theta\right)};q\right)_{\infty}}
\]
 and a sequence of positive numbers $\left\{ \mu_{n}\right\} _{n\in\mathbb{N}}$
such that 
\[
\mu_{m}-\mu_{n}\ge\frac{\alpha}{\sqrt{\log q^{-1/2}}}\log^{1/2}\left(1+\left|m-n\right|\right).
\]
 Then by \cite[(5.24)]{IsmailZhang1},
\[
\begin{aligned} & f_{n}=\int_{-\infty}^{\infty}f(x)e^{-i\mu_{n}x}\psi\left(x\vert\frac{1}{4}\right)dx=\int_{-\infty}^{\infty}\frac{\exp\left(\frac{y^{2}}{\log q}-iy\mu_{n}\right)dy}{\left(te^{i\left(y+\theta\right)},te^{i\left(y-\theta\right)};q\right)_{\infty}\sqrt{\pi\log q^{-1}}}\\
 & =q^{\mu_{n}^{2}/4}\left(t^{2}q^{1-\mu_{n}};q^{2}\right)_{\infty}\mathcal{E}_{q}\left(x;tq^{-\mu_{n}/2}\right).
\end{aligned}
\]
By \cite[(2.14)]{IsmailZhang1} and \cite[(5.24)]{IsmailZhang1}
\[
\begin{aligned} & \int_{-\infty}^{\infty}\left|f(x)\right|^{2}\psi\left(x\vert\frac{1}{4}\right)dx=\int_{-\infty}^{\infty}\frac{\exp\left(\frac{y^{2}}{\log q}\right)\left(\pi\log q^{-1}\right)^{-\frac{1}{2}}dy}{\left(te^{i\left(y+\theta\right)},te^{i\left(y-\theta\right)},\overline{t}e^{-i\left(y+\theta\right)},\overline{t}e^{-i\left(y-\theta\right)};q\right)_{\infty}}\\
 & =\sum_{n=0}^{\infty}\frac{H_{n}(\cos\theta\vert q)\overline{t}^{n}}{(q;q)_{n}}\int_{-\infty}^{\infty}\frac{\exp\left(\frac{y^{2}}{\log q}-iny\right)\left(\pi\log q^{-1}\right)^{-\frac{1}{2}}dy}{\left(te^{i\left(y+\theta\right)},te^{i\left(y-\theta\right)};q\right)_{\infty}}\\
 & =\sum_{n=0}^{\infty}\frac{H_{n}(x\vert q)\overline{t}^{n}}{(q;q)_{n}}q^{n^{2}/4}\left(t^{2}q^{1-n};q^{2}\right)_{\infty}\mathcal{E}_{q}\left(x;tq^{-n/2}\right)\\
 & =\sum_{n=0}^{\infty}\frac{H_{n}(x\vert q)t^{n}}{(q;q)_{n}}q^{n^{2}/4}\left(\overline{t}^{2}q^{1-n};q^{2}\right)_{\infty}\mathcal{E}_{q}\left(x;\overline{t}q^{-n/2}\right),
\end{aligned}
\]
where $H_{n}(x\vert q)$ is the $n$th $q$-Hermite polynomial and
$\mathcal{E}_{q}\left(x;t\right)$ is a $q$-analogue of the plane
wave function, \cite{Ismail,IsmailZhang1}. 

Then,
\begin{equation}
\sum_{n=0}^{\infty}\frac{H_{n}(x\vert q)\overline{t}^{n}}{(q;q)_{n}}q^{n^{2}/4}\left(t^{2}q^{1-n};q^{2}\right)_{\infty}\mathcal{E}_{q}\left(x;tq^{-n/2}\right)\ge0.\label{eq:3.21}
\end{equation}
By the generalized Bessel inequality,
\begin{equation}
\begin{aligned} & \sum_{n=1}^{\infty}q^{\mu_{n}^{2}/2}\left|\left(t^{2}q^{1-\mu_{n}};q^{2}\right)_{\infty}\mathcal{E}_{q}\left(x;tq^{-\mu_{n}/2}\right)\right|^{2}\\
 & \le\left(\sup_{m\in\mathbb{N}}\sum_{n=1}^{\infty}q^{(\mu_{m}-\mu_{n})^{2}/4}\right)\sum_{n=0}^{\infty}\frac{H_{n}(x\vert q)t^{n}}{(q;q)_{n}}q^{n^{2}/4}\left(\overline{t}^{2}q^{1-n};q^{2}\right)_{\infty}\mathcal{E}_{q}\left(x;\overline{t}q^{-n/2}\right).
\end{aligned}
\label{eq:3.22}
\end{equation}
 
\end{example}

\begin{example}
For $\Re(\nu)>-\frac{1}{2},\ \alpha>\sqrt{2},\ z\in\mathbb{C}\backslash\left\{ 0\right\} $
let 
\[
\beta=\frac{1}{2},\quad f(x)=\frac{\left(\frac{q^{\nu+1/2}z^{2}e^{ix}}{4};q\right)_{\infty}}{\left(q,-q^{\nu+1/2}e^{ix};q\right)_{\infty}}
\]
 and any sequence of positive numbers $\left\{ \mu_{n}\right\} _{n\in\mathbb{N}}$
such that 
\[
\mu_{m}-\mu_{n}\ge\frac{\alpha}{\sqrt{\log q^{-1}}}\log^{1/2}\left(1+\left|m-n\right|\right).
\]
Then by \cite[(5.68)]{IsmailZhang1}
\[
\begin{aligned} & f_{n}=\int_{-\infty}^{\infty}f(x)e^{-i\mu_{n}x}\psi\left(x\vert\frac{1}{2}\right)dx=\int_{-\infty}^{\infty}\frac{\left(\frac{q^{\nu+1/2}z^{2}e^{ix}}{4};q\right)_{\infty}\exp\left(\frac{x^{2}}{\log q^{2}}-i\mu_{n}x\right)}{\left(q,-q^{\nu+1/2}e^{ix};q\right)_{\infty}\sqrt{2\pi\log q^{-1}}}dx\\
 & =q^{\mu_{n}^{2}/2}J_{\nu-\mu_{n}}^{(2)}\left(z;q\right)\left(\frac{z}{2}\right)^{\mu_{n}-\nu},
\end{aligned}
\]
where $J_{\nu}^{(2)}\left(z;q\right)$ is $q$-Bessel function of
the second kind, \cite{Andrews,DLMF,Ismail}.

By the $q$-binomial theorem \cite{Andrews,DLMF,Ismail}
\[
\begin{aligned} & \int_{-\infty}^{\infty}\left|f(x)\right|^{2}\psi\left(x\vert\frac{1}{2}\right)dx=\int_{-\infty}^{\infty}\frac{\left(\frac{q^{\overline{\nu}+1/2}\overline{z}^{2}e^{-ix}}{4};q\right)_{\infty}}{\left(q,-q^{\overline{\nu}+1/2}e^{-ix};q\right)_{\infty}}\frac{\left(\frac{q^{\nu+1/2}z^{2}e^{ix}}{4};q\right)_{\infty}\exp\left(\frac{x^{2}}{\log q^{2}}\right)}{\left(q,-q^{\nu+1/2}e^{ix};q\right)_{\infty}\sqrt{2\pi\log q^{-1}}}dx\\
 & =\frac{1}{(q;q)_{\infty}}\sum_{n=0}^{\infty}\frac{\left(-\frac{\overline{z}^{2}}{4};q\right)_{n}}{(q;q)_{n}}\left(-q^{\overline{\nu}+1/2}\right)^{n}\int_{-\infty}^{\infty}\frac{\left(\frac{q^{\nu+1/2}z^{2}e^{ix}}{4};q\right)_{\infty}\exp\left(\frac{x^{2}}{\log q^{2}}-inx\right)}{\left(q,-q^{\nu+1/2}e^{ix};q\right)_{\infty}\sqrt{2\pi\log q^{-1}}}dx\\
 & =\frac{1}{(q;q)_{\infty}}\sum_{n=0}^{\infty}\frac{\left(-\frac{\overline{z}^{2}}{4};q\right)_{n}}{(q;q)_{n}}\left(-q^{\overline{\nu}+1/2}\right)^{n}q^{n^{2}/2}J_{\nu-n}^{(2)}\left(z;q\right)\left(\frac{z}{2}\right)^{n-\nu}\\
 & =\frac{1}{(q;q)_{\infty}}\sum_{n=0}^{\infty}\frac{\left(-\frac{z^{2}}{4};q\right)_{n}}{(q;q)_{n}}\left(-q^{\nu+1/2}\right)^{n}q^{n^{2}/2}J_{\overline{\nu}-n}^{(2)}\left(\overline{z};q\right)\left(\frac{z}{2}\right)^{n-\overline{\nu}}.
\end{aligned}
\]
Thus,
\begin{equation}
\begin{aligned} & \sum_{n=0}^{\infty}\frac{\left(-\frac{\overline{z}^{2}}{4};q\right)_{n}}{(q;q)_{n}}\left(-q^{\overline{\nu}+1/2}\right)^{n}q^{n^{2}/2}J_{\nu-n}^{(2)}\left(z;q\right)\left(\frac{z}{2}\right)^{n-\nu}\\
 & =\sum_{n=0}^{\infty}\frac{\left(-\frac{z^{2}}{4};q\right)_{n}}{(q;q)_{n}}\left(-q^{\nu+1/2}\right)^{n}q^{n^{2}/2}J_{\overline{\nu}-n}^{(2)}\left(\overline{z};q\right)\left(\frac{z}{2}\right)^{n-\overline{\nu}}\ge0
\end{aligned}
\label{eq:3.23}
\end{equation}
and
\begin{equation}
\begin{aligned} & \sum_{n=1}^{\infty}q^{\mu_{n}^{2}}\left|J_{\nu-\mu_{n}}^{(2)}\left(z;q\right)\left(\frac{z}{2}\right)^{\mu_{n}-\nu}\right|^{2}\\
 & \le\frac{\sup_{m\in\mathbb{N}}\sum_{n=1}^{\infty}q^{(\mu_{m}-\mu_{n})^{2}/2}}{(q;q)_{\infty}}\sum_{n=0}^{\infty}\frac{\left(-\frac{z^{2}}{4};q\right)_{n}}{(q;q)_{n}}\left(-q^{\nu+1/2}\right)^{n}q^{n^{2}/2}J_{\overline{\nu}-n}^{(2)}\left(\overline{z};q\right)\left(\frac{z}{2}\right)^{n-\overline{\nu}}.
\end{aligned}
\label{eq:3.24}
\end{equation}
 
\end{example}

\subsection{Gamma distribution}

For $\sigma>0$ let 

\begin{equation}
\phi_{n}(x)=x^{i\mu_{n}},\quad\omega_{1}(x)=\frac{e^{-x}x^{\sigma-1}1_{(0,\infty)}(x)}{\Gamma(\sigma)},\label{eq:3.25}
\end{equation}
where $\left\{ \mu_{n}\right\} _{n=1}^{\infty}$ is a sequence of
nondecreasing positive numbers such that for $m\ge n$

\begin{equation}
\mu_{m}-\mu_{n}\ge c_{1}\log\left(1+\left|m-n\right|\right),\label{eq:3.26}
\end{equation}
and $c_{1}$ is a positive number satisfying $\frac{c_{1}\pi}{2}>1$.

Then
\begin{equation}
a_{m,n}=\int_{0}^{\infty}x^{i(\mu_{m}-\mu_{n})}d\omega_{1}(x)=\frac{\Gamma\left(\sigma+i(\mu_{m}-\mu_{n})\right)}{\Gamma(\sigma)}.\label{eq:3.27}
\end{equation}
Since $\frac{c_{1}\pi}{2}>1$, then there exists a positive number
$\epsilon$ such that $c_{1}\left(\frac{\pi}{2}-\epsilon\right)>1$,
by \cite[(5.11.9)]{DLMF}, 
\begin{equation}
\left|\Gamma(x+iy)\right|=\mathcal{O}\left(e^{-(\pi/2-\epsilon)\left|y\right|}\right)\label{eq:3.28}
\end{equation}
as $y\to\pm\infty$, uniformly for $x$ on any compact subset of $\mathbb{R}$.
Hence,
\begin{equation}
\begin{aligned} & \sum_{n=1}^{\infty}\left|\frac{\Gamma(\sigma+i(\mu_{m}-\mu_{n})}{\Gamma(\sigma)}\right|=\mathcal{O}\left(\sum_{n=1}^{\infty}e^{-(\pi/2-\epsilon)\left|\mu_{m}-\mu_{n}\right|}\right)\\
 & =\mathcal{O}\left(\sum_{n=1}^{\infty}e^{-(\pi/2-\epsilon)c_{1}\log\left(1+\left|m-n\right|\right)}\right)=\mathcal{O}\left(\sum_{n=-\infty}^{\infty}\frac{1}{\left(1+\left|n\right|\right)^{(\pi/2-\epsilon)c_{1}}}\right),
\end{aligned}
\label{eq:3.29}
\end{equation}
which shows
\begin{equation}
C_{1}=\frac{1}{\left|\Gamma(\sigma)\right|}\sup_{m\in\mathbb{N}}\sum_{n=1}^{\infty}\left|\Gamma\left(\sigma+i(\mu_{m}-\mu_{n})\right)\right|<\infty.\label{eq:3.30}
\end{equation}
Then, the infinite matrix
\begin{equation}
A=\left(\frac{\Gamma\left(\sigma+i(\mu_{m}-\mu_{n})\right)}{\Gamma(\sigma)}\right)_{m,n=1}^{\infty}\label{eq:3.31}
\end{equation}
defines a positive semidefinite operators on $\ell^{2}(\mathbb{N})$
with its operator norm less than $C_{1}$. 
\begin{example}
Let $f(x)=\log x$, then by \cite[pp40,  (4.44)]{Oberhettinger2}
\[
f_{n}=\int_{0}^{\infty}x^{-i\mu_{n}}\log x\omega_{1}(x)dx=\frac{\Gamma(\sigma-i\mu_{n})\psi(\sigma-i\mu_{n})}{\Gamma(\sigma)},
\]
and by \cite[pp39, (4.43)]{Oberhettinger2}
\[
\int_{0}^{\infty}\log^{2}x\omega_{1}(x)dx=\left(\psi^{2}(\sigma)+\psi'(\sigma)\right),
\]
where
\[
\psi(z)=\frac{d}{dz}\log\Gamma(z)=\frac{\Gamma'(z)}{\Gamma(z)}.
\]
 By \eqref{eq:2.15} we have
\begin{equation}
\begin{aligned} & \sum_{n=1}^{\infty}\left|\Gamma(\sigma+i\mu_{n})\psi(\sigma+i\mu_{n})\right|^{2}\\
 & \le\Gamma(\sigma)\left(\psi^{2}(\sigma)+\psi'(\sigma)\right)\cdot\sup_{m\in\mathbb{N}}\sum_{n=1}^{\infty}\left|\Gamma\left(\sigma+i(\mu_{m}-\mu_{n})\right)\right|.
\end{aligned}
\label{eq:3.32}
\end{equation}
 
\end{example}

\begin{example}
For $\ell\ge0$ let $f(x)=L_{\ell}^{(\sigma-1)}(x)$ then by \cite[pp89,  (9.31)]{Oberhettinger2},
\[
\begin{aligned} & f_{n}=\int_{0}^{\infty}L_{\ell}^{(\sigma-1)}(x)x^{-i\mu_{n}}\omega_{1}(x)dx\\
 & =\frac{\Gamma(\sigma-i\mu_{n})\Gamma(\ell+i\mu_{n})}{\ell!\Gamma(\sigma)\Gamma(i\mu_{n})},
\end{aligned}
\]
and by the orthogonality of Laguerre polynomials, \cite{Andrews,DLMF,Ismail}
\[
\int_{0}^{\infty}\left(L_{\ell}^{(\sigma-1)}(x)\right)^{2}\omega_{1}(x)dx=\frac{\Gamma(\ell+\sigma)}{\ell!\Gamma(\sigma)}.
\]
 Then,
\begin{equation}
\begin{aligned} & \sum_{n=1}^{\infty}\left|\frac{\Gamma(\sigma+i\mu_{n})\Gamma(\ell+i\mu_{n})}{\Gamma(i\mu_{n})}\right|^{2}\\
 & \le\frac{\Gamma(\ell+\sigma)}{\ell!}\sup_{m\in\mathbb{N}}\sum_{n=1}^{\infty}\left|\Gamma\left(\sigma+i(\mu_{m}-\mu_{n})\right)\right|.
\end{aligned}
\label{eq:3.33}
\end{equation}
 
\end{example}

\begin{example}
Let $f(x)=\frac{e^{x}}{e^{x}+1}$ then by \cite[(25.5.3)]{DLMF},
\[
f_{n}=\int_{0}^{\infty}\frac{x^{-i\mu_{n}}\omega_{1}(x)dx}{e^{-x}+1}=\frac{\left(1-2^{1-\sigma+i\mu_{n}}\right)\Gamma(\sigma-i\mu_{n})\zeta(\sigma-i\mu_{n})}{\Gamma(\sigma)}
\]
and by \cite[(25.5.4)]{DLMF},
\[
\int_{0}^{\infty}\frac{\omega_{1}(x)dx}{(e^{-x}+1)^{2}}=\frac{\left(1-2^{1-\sigma}\right)\Gamma(\sigma+1)\zeta(\sigma)}{\Gamma(\sigma)}.
\]
 Then
\begin{equation}
\begin{aligned} & \sum_{n=1}^{\infty}\left|\left(1-2^{1-\sigma+i\mu_{n}}\right)\Gamma(\sigma+i\mu_{n})\zeta(\sigma+i\mu_{n})\right|^{2}\\
 & =2^{\sigma-1}\Gamma(\sigma+1)\left(1-2^{1-\sigma}\right)\zeta(\sigma)\sup_{m\in\mathbb{N}}\sum_{n=1}^{\infty}\left|\Gamma\left(\sigma+i(\mu_{m}-\mu_{n})\right)\right|.
\end{aligned}
\label{eq:3.34}
\end{equation}
\end{example}

\begin{example}
Let $\nu\ge0$ and $f(x)=e^{x/2}K_{\nu}(x/2)$, then by \cite[pp115, (11.5)]{Oberhettinger2},
\[
\begin{aligned} & f_{n}=\int_{0}^{\infty}K_{\nu}(x/2)x^{-i\mu_{n}}\omega_{1}(x)dx=\\
 & =\frac{\sqrt{\pi}\Gamma(\sigma-i\mu_{n}-\nu)\Gamma(\sigma-i\mu_{n}+\nu)}{\Gamma(\sigma)\Gamma(\sigma-i\mu_{n}+1/2)}
\end{aligned}
\]
 and by \cite[pp123, (11.45)]{Oberhettinger2},
\[
\begin{aligned} & \int_{0}^{\infty}\left|f(x)\right|^{2}\omega_{1}(x)dx=\frac{1}{\Gamma(\sigma)}\int_{0}^{\infty}x^{\sigma-1}K_{\nu}^{2}(x/2)dx\\
 & =\frac{\sqrt{\pi}\Gamma\left(\sigma/2+\nu\right)\Gamma\left(\sigma/2-\nu\right)\Gamma(\sigma/2)}{2^{2-\sigma}\Gamma\left((\sigma+1)/2\right)\Gamma(\sigma)}.
\end{aligned}
\]
 Hence, for $\sigma,\nu>0$, 
\begin{equation}
\begin{aligned} & \sum_{n=1}^{\infty}\left|\frac{\Gamma(\sigma+i\mu_{n}-\nu)\Gamma(\sigma+i\mu_{n}+\nu)}{\Gamma(\sigma+i\mu_{n}+1/2)}\right|^{2}\\
 & \le\frac{\Gamma\left(\sigma/2+\nu\right)\Gamma\left(\sigma/2-\nu\right)\Gamma(\sigma/2)}{2^{2-\sigma}\sqrt{\pi}\Gamma\left((\sigma+1)/2\right)}\sup_{m\in\mathbb{N}}\sum_{n=1}^{\infty}\left|\Gamma\left(\sigma+i(\mu_{m}-\mu_{n})\right)\right|.
\end{aligned}
\label{eq:3.35}
\end{equation}
\end{example}

\begin{example}
Let $f(x)=e^{-a^{2}/(4x)},\ a>0$, then by \cite[(10.32.10)]{DLMF}

\[
f_{n}=\int_{0}^{\infty}f(x)x^{-i\mu_{n}}\omega_{1}(x)dx=\int_{0}^{\infty}\frac{e^{-x-a^{2}/(4x)}dx}{\Gamma(\sigma)x^{-\sigma+i\mu_{n}+1}}=\frac{a^{-\sigma+i\mu_{n}}K_{\sigma-i\mu_{n}}(a)}{2^{\sigma-i\mu_{n}-1}\Gamma(\sigma)}
\]
and
\[
\int_{0}^{\infty}\left|f(x)\right|^{2}\omega_{1}(x)dx=\int_{0}^{\infty}\frac{e^{-x-a^{2}/(2x)}dx}{x^{1-\sigma}\Gamma(\sigma)}=\frac{2^{\sigma/2+1}a^{\sigma}K_{\sigma}(\sqrt{2}a)}{\Gamma(\sigma)},
\]
where $K_{\nu}(z)$ is the modified Bessel function, \cite{Andrews,DLMF,Ismail,Rademacher}. 

Then,
\begin{equation}
\sum_{n=1}^{\infty}\left|K_{\sigma+i\mu_{n}}(a)\right|^{2}\le a^{3\sigma}2^{5\sigma/2-1}K_{\sigma}(\sqrt{2}a)\sup_{m\in\mathbb{N}}\sum_{n=1}^{\infty}\left|\Gamma\left(\sigma+i(\mu_{m}-\mu_{n})\right)\right|.\label{eq:3.36}
\end{equation}
\end{example}

\begin{example}
For $a,\nu>0$ let $f(x)=J_{\nu}(a\sqrt{x})$. Then by \cite[pp138,  (14.28)]{Oberhettinger3}
we get
\[
\begin{aligned} & f_{n}=\int_{0}^{\infty}f(x)x^{-i\mu_{n}}\omega_{1}(x)dx=\frac{1}{\Gamma(\sigma)}\int_{0}^{\infty}J_{\nu}(a\sqrt{x})\exp\left(-x\right)t^{\sigma-i\mu_{n}-1}dt\\
 & =\frac{(a/2)^{\nu}\Gamma\left(\sigma-i\mu_{n}+\frac{\nu}{2}\right)}{\Gamma(\nu+1)\Gamma(\sigma)}{}_{1}F_{1}\left(\sigma-i\mu_{n}+\frac{\nu}{2};\nu+1;-a^{2}\right),
\end{aligned}
\]
and by \cite[pp140,  14.35]{Oberhettinger3} 
\[
\begin{aligned} & \int_{0}^{\infty}f^{2}(x)\omega_{1}(x)dx=\frac{1}{\Gamma(\sigma)}\int_{0}^{\infty}x^{\sigma-1}\exp\left(-x\right)J_{\nu}^{2}(a\sqrt{x})dx\\
 & =\frac{a^{2\nu}\Gamma\left(\sigma+\nu\right)}{2^{2\nu}\Gamma(\sigma)\Gamma^{2}(\nu+1)}{}_{2}F_{2}\left(\nu+\frac{1}{2},\sigma+\nu;\nu+1,2\nu+1;-a^{2}\right).
\end{aligned}
\]
 Then by \eqref{eq:2.15} we get
\begin{equation}
\begin{aligned} & \sum_{n=1}^{\infty}\left|\Gamma\left(\sigma-i\mu_{n}+\frac{\nu}{2}\right){}_{1}F_{1}\left(\sigma-i\mu_{n}+\frac{\nu}{2};\nu+1;-a^{2}\right)\right|^{2}\\
 & \le{}_{2}F_{2}\left(\nu+\frac{1}{2},\sigma+\nu;\nu+1,2\nu+1;-a^{2}\right)\\
 & \times\Gamma\left(\sigma+\nu\right)\sup_{m\in\mathbb{N}}\sum_{n=1}^{\infty}\left|\Gamma\left(\sigma+i(\mu_{m}-\mu_{n})\right)\right|.
\end{aligned}
\label{eq:3.37}
\end{equation}
\end{example}

\begin{example}
Let $f(t)=_{1}F_{1}\left(a;b;xt\right),\ x\in(-1,1)$, then by \cite[pp115,  exercise 11]{Andrews}
we have 
\[
\begin{aligned} & f_{n}=\int_{0}^{\infty}f(t)t^{-i\mu_{n}}\omega_{1}(t)dt=\frac{1}{\Gamma(\sigma)}\int_{0}^{\infty}e^{-t}t^{\sigma-i\mu_{n}-1}{}_{1}F_{1}\left(a;\sigma;xt\right)dt\\
 & =_{2}F_{1}(a,\sigma-i\mu_{n};\sigma;x)
\end{aligned}
\]
and by \cite[pp235, exercise 6]{Andrews} 
\[
\begin{aligned} & \int_{0}^{\infty}f^{2}(t)\omega_{1}(t)dt=\frac{1}{\Gamma(\sigma)}\int_{0}^{\infty}e^{-t}t^{\sigma-1}{}_{1}F_{1}\left(a;\sigma;tx\right){}_{1}F_{1}\left(a;\sigma;tx\right)dt\\
 & =\frac{x^{2\sigma}}{(1-x)^{2a}}{}_{2}F_{1}\left(a,a;\sigma;\frac{x^{2}}{(1-x)^{2}}\right).
\end{aligned}
\]
 Then 
\begin{equation}
\begin{aligned} & \sum_{n=1}^{\infty}\left|_{2}F_{1}(a,\sigma-i\mu_{n};\sigma;x)\right|^{2}\\
 & \le x^{2\sigma}{}_{2}F_{1}\left(a,a;\sigma;\frac{x^{2}}{(1-x)^{2}}\right)\frac{\sup_{m\in\mathbb{N}}\sum_{n=1}^{\infty}\left|\Gamma\left(\sigma+i(\mu_{m}-\mu_{n})\right)\right|}{\Gamma(\sigma)(1-x)^{2a}}.
\end{aligned}
\label{eq:3.38}
\end{equation}
\end{example}

\subsection{Beta distribution}

For $p,q>0$ let 
\begin{equation}
\omega_{2}(x)=\frac{x^{p-1}(1-x)^{q-1}1_{(0,1)}(x)}{B(p,q)},\label{eq:3.39}
\end{equation}
where $B(p,q)$ is the Euler Beta function, \cite[(10.32.9)]{DLMF}.
Let $\phi_{n}(x)=x^{i\lambda_{n}}$ and 
\begin{equation}
a_{m,n}=\int_{0}^{1}\frac{x^{p+i(\lambda_{m}-\lambda_{n})-1}(1-x)^{q-1}dx}{B(p,q)}=\frac{B\left(p+i(\lambda_{m}-\lambda_{n}),q\right)}{B(p,q)},\label{eq:3.40}
\end{equation}
where $\left\{ \lambda_{n}\right\} _{n=1}^{\infty}$ is a sequence
of positive numbers. If there exist two positive numbers $\alpha,\beta>0$
with $\beta q>1$ such that for $m\ge n\ge1$,
\begin{equation}
\lambda_{m}-\lambda_{n}\ge\alpha\left(1+\left|m-n\right|\right)^{\beta}.\label{eq:3.41}
\end{equation}
Then by \cite[(5.11.12)]{DLMF}
\[
\begin{aligned} & \frac{B\left(p+i(\lambda_{m}-\lambda_{n}),q\right)}{B(p,q)}=\frac{\Gamma(p+i(\lambda_{m}-\lambda_{n}))}{\Gamma(p)}\frac{\Gamma(p+q)}{\Gamma(p+q+i(\lambda_{m}-\lambda_{n}))}\\
 & =\mathcal{O}\left(\frac{1}{(\lambda_{m}-\lambda_{n})^{q}}\right)=\mathcal{O}\left(\frac{1}{(1+\left|m-n\right|)^{\beta q}}\right),
\end{aligned}
\]
 and 
\[
\sum_{n=1}^{\infty}\left|B\left(p+i(\lambda_{m}-\lambda_{n}),q\right)\right|=\mathcal{O}\left(\sum_{n=1}^{\infty}\frac{1}{(1+\left|m-n\right|)^{\beta q}}\right)=\mathcal{O}\left(\sum_{n=1}^{\infty}\frac{1}{(1+\left|n\right|)^{\beta q}}\right).
\]
Hence,
\begin{equation}
C_{2}=\frac{1}{B(p,q)}\sup_{m\in\mathbb{N}}\sum_{n=1}^{\infty}\left|B\left(p+i(\lambda_{m}-\lambda_{n}),q\right)\right|<\infty.\label{eq:3.42}
\end{equation}
Then the infinite matrix 
\begin{equation}
A=\left(\frac{B\left(p+i(\lambda_{m}-\lambda_{n}),q\right)}{B(p,q)}\right)_{m,n=1}^{\infty}\label{eq:3.43}
\end{equation}
defines a bounded positive semidefinite operator on $\ell^{2}\left(\mathbb{N}\right)$
with operator norm $\left\Vert A\right\Vert _{2}\le C_{2}$. 
\begin{example}
Let $f(x)=\log x$, then by \cite[pp39,  (4.41)]{Oberhettinger2},
\[
\begin{aligned} & f_{n}=\int_{0}^{1}f(x)x^{-i\lambda_{n}}\omega_{2}(x)dx=\int_{0}^{1}\frac{x^{p-i\lambda_{n}-1}(1-x)^{q-1}\log xdx}{B(p,q)}\\
 & =\frac{B(p-i\lambda_{n},q)}{B(p,q)}\left(\psi(p-i\lambda_{n})-\psi(p+q-i\lambda_{n})\right)
\end{aligned}
\]
 and by \cite[pp39,  (4.42)]{Oberhettinger2},
\[
\begin{aligned} & \int_{0}^{1}f^{2}(x)\omega_{2}(x)dx=\int_{0}^{1}\frac{x^{p-1}(1-x)^{q-1}\log^{2}xdx}{B(p,q)}\\
 & =(\psi(p)-\psi(q+p))^{2}+\psi'(p)-\psi'(p+q).
\end{aligned}
\]
 By the generalized Bessel inequality \eqref{eq:2.1}, 
\begin{equation}
\begin{aligned} & \sum_{n=1}^{\infty}\left|B(p+i\lambda_{n},q)\left(\psi(p+i\lambda_{n})-\psi(p+q+i\lambda_{n})\right)\right|^{2}\\
 & \le\left((\psi(p)-\psi(q+p))^{2}+\psi'(p)-\psi'(p+q)\right)\\
 & \times B(p,q)\sup_{m\in\mathbb{N}}\sum_{n=1}^{\infty}\left|B\left(p+i(\lambda_{m}-\lambda_{n}),q\right)\right|.
\end{aligned}
\label{eq:3.44}
\end{equation}
\end{example}

\begin{example}
Let 
\[
f(x)=(1-zx)^{-a},\quad0<z<1,\ a\in\mathbb{C}.
\]
By the Euler's integral representation for the Gauss hypergeometric
function $_{2}F_{1}$, \cite[ pp13,  (1.4.8)]{Ismail}
\[
\begin{aligned} & f_{n}=\int_{0}^{1}f(x)x^{-i\lambda_{n}}\omega_{2}(x)dx=\int_{0}^{1}\frac{x^{p-i\lambda_{n}-1}(1-x)^{q-1}(1-zx)^{-a}dx}{B(p,q)}\\
 & =\frac{B(p-i\lambda_{n},q)}{B(p,q)}{}_{2}F_{1}(a,p-i\lambda_{n};p+q-i\lambda_{n};z)
\end{aligned}
\]
and
\[
\begin{aligned} & \int_{0}^{1}\left|f(x)\right|^{2}\omega_{2}(x)dx=\int_{0}^{1}\frac{x^{p-1}(1-x)^{q-1}(1-zx)^{-2\Re(a)}dx}{B(p,q)}\\
 & =_{2}F_{1}(2\Re(a),p;p+q;z).
\end{aligned}
\]
Then,
\begin{equation}
\begin{aligned} & \sum_{n=1}^{\infty}\left|B(p-i\lambda_{n},q){}_{2}F_{1}(a,p-i\lambda_{n};p+q-i\lambda_{n};z)\right|^{2}\\
 & \le B(p,q){}_{2}F_{1}(2\Re(a),p;p+q;z)\sup_{m\in\mathbb{N}}\sum_{n=1}^{\infty}\left|B\left(p+i(\lambda_{m}-\lambda_{n}),q\right)\right|.
\end{aligned}
\label{eq:3.45}
\end{equation}
\end{example}

\begin{example}
For $\ell\in\mathbb{N}_{0}$ let 
\[
f(x)=P_{\ell}^{(p-1,q-1)}(2x-1),
\]
where $P_{n}^{(\alpha,\beta)}$ is the $n$-th Jacobi polynomial,
\cite{Andrews,DLMF,Ismail}. Then by \cite[pp91, (9.44)]{Oberhettinger2},
\[
\begin{aligned} & f_{n}=\int_{0}^{1}f(x)x^{-i\lambda_{n}}\omega_{2}(x)dx=\int_{0}^{1}\frac{P_{\ell}^{(p-1,q-1)}(2x-1)x^{p-i\lambda_{n}-1}(1-x)^{q-1}dx}{B(p,q)}\\
 & =\frac{\Gamma(\ell+q)B(p-i\lambda_{n},q)}{(-1)^{\ell}\ell!\Gamma(q)B(p,q)}{}_{3}F_{2}\left(-\ell,\ell+p+q-1,p-i\lambda_{n};q,p+q-i\lambda_{n};1\right).
\end{aligned}
\]
By the orthogonality of the Jacobi polynomials we have \cite{Andrews,DLMF,Ismail}
\[
\begin{aligned} & \int_{0}^{1}f^{2}(x)\omega_{2}(x)dx=\int_{0}^{1}\frac{\left(P_{\ell}^{(p-1,q-1)}(2x-1)\right)^{2}x^{p-1}(1-x)^{q-1}dx}{B(p,q)}\\
 & =\frac{1}{2\ell+p+q-1}\frac{\Gamma(\ell+p)\Gamma(\ell+q)}{\ell!\Gamma(\ell+p+q-1)B(p,q)}.
\end{aligned}
\]
Therefore, 
\begin{equation}
\begin{aligned} & \sum_{n=1}^{\infty}\left|B(p-i\lambda_{n},q){}_{3}F_{2}\left(\begin{array}{cc}
\begin{array}{c}
-\ell,\ell+p+q-1,p-i\lambda_{n}\\
q,p+q-i\lambda_{n}
\end{array} & \bigg|1\end{array}\right)\right|^{2}\\
 & \le\frac{\ell!\Gamma(\ell+p)\Gamma^{2}(q)\sup_{m\in\mathbb{N}}\sum_{n=1}^{\infty}\left|B\left(p+i(\lambda_{m}-\lambda_{n}),q\right)\right|}{(2\ell+p+q-1)\Gamma(\ell+p+q-1)\Gamma(\ell+q)}.
\end{aligned}
\label{eq:3.46}
\end{equation}
\end{example}

\begin{example}
Let $p,q,p+\nu>0$ and $f(x)=J_{2\nu}\left(x^{1/2}\right)$. Then
by \cite[pp107, (10.61)]{Oberhettinger2} to get
\[
\begin{aligned} & \int_{0}^{1}f(x)x^{-i\lambda_{n}}\omega_{2}(x)dx=\int_{0}^{1}\frac{x^{p-i\lambda_{n}-1}(1-x)^{q-1}J_{2\nu}\left(x^{1/2}\right)dx}{B(p,q)}\\
 & =\frac{B\left(q,p-i\lambda_{n}+\nu\right)}{4^{\nu}\Gamma(2\nu+1)B(p,q)}{}_{1}F_{2}\left(\begin{array}{cc}
\begin{array}{c}
p-i\lambda_{n}+\nu\\
2\nu+1,p+q-i\lambda_{n}+\nu
\end{array} & \bigg|-\frac{1}{4}\end{array}\right)
\end{aligned}
\]
and by \cite[pp109,  (10.68)]{Oberhettinger2} to get 
\[
\begin{aligned} & \int_{0}^{1}f^{2}(x)\omega_{2}(x)dx=\int_{0}^{1}\frac{x^{p-1}(1-x)^{q-1}J_{2\nu}^{2}(x^{1/2})dx}{B(p,q)}\\
 & =\frac{16^{-\nu}B(p+2\nu,q)}{\Gamma^{2}(1+2\nu)B(p,q)}{}_{2}F_{3}\left(\begin{array}{cc}
\begin{array}{c}
p+2\nu,2\nu+1/2\\
p+q+2\nu,2\nu+1,4\nu+1
\end{array} & \bigg|\end{array}-1\right).
\end{aligned}
\]
Therefore,
\begin{equation}
\begin{aligned} & \sum_{n=1}^{\infty}\left|B\left(q,p+i\lambda_{n}+\nu\right){}_{1}F_{2}\left(\begin{array}{cc}
\begin{array}{c}
p+i\lambda_{n}+\nu\\
2\nu+1,p+q+i\lambda_{n}+\nu
\end{array} & \bigg|-\frac{1}{4}\end{array}\right)\right|^{2}\\
 & \le B(p+2\nu,q)\sup_{m\in\mathbb{N}}\sum_{n=1}^{\infty}\left|B\left(p+i(\lambda_{m}-\lambda_{n}),q\right)\right|\\
 & \times{}_{2}F_{3}\left(\begin{array}{cc}
\begin{array}{c}
p+2\nu,2\nu+1/2\\
p+q+2\nu,2\nu+1,4\nu+1
\end{array} & \bigg|\end{array}-1\right).
\end{aligned}
\label{eq:3.47}
\end{equation}
\end{example}

\section{Applications of Corollary \ref{cor:5}}

\subsection{Case: $\psi(k)=-\frac{\zeta(\sigma)}{\zeta^{\prime}(\sigma)}\frac{\Lambda(k)}{k^{\sigma}}$ }

The von Mangoldt function $\Lambda(k)$ has a generating function,
\cite[(27.2.14)]{DLMF},\cite{Apostol,HardyWright}

\begin{equation}
{\displaystyle \frac{\zeta^{\prime}(s)}{\zeta(s)}=-\sum_{n=1}^{\infty}\frac{\Lambda(n)}{n^{s}},\quad\Re(s)>1,}\label{eq:4.1}
\end{equation}
where $\zeta(s)$ is the Riemann zeta function, \cite[(25.2.1)]{DLMF}.
More generally, for any completely multiplicative function \LyXThinSpace $f(n)$
if 

\begin{equation}
{\displaystyle F(s)=\sum_{n=1}^{\infty}\frac{f(n)}{n^{s}},\quad\Re(s)>\sigma_{0},}\label{eq:4.2}
\end{equation}
then,

\begin{equation}
{\displaystyle \frac{F^{\prime}(s)}{F(s)}=-\sum_{n=1}^{\infty}\frac{f(n)\Lambda(n)}{n^{s}},\quad\Re(s)>\sigma_{0}.}\label{eq:4.3}
\end{equation}
Hence, for a Dirichlet character $\chi$, 
\begin{equation}
L(s,\chi)=\sum_{n=1}^{\infty}\frac{\chi(n)}{n^{s}},\quad-\frac{L'(s,\chi)}{L(s,\chi)}=\sum_{n=1}^{\infty}\frac{\chi(n)\Lambda(n)}{n^{s}},\label{eq:4.4}
\end{equation}
and for the Liouville function $\lambda(n)$, 
\begin{equation}
\frac{\zeta(2s)}{\zeta(s)}=\sum_{n=1}^{\infty}\frac{\lambda(n)}{n^{s}},\quad\frac{\zeta'(s)}{\zeta(s)}-2\frac{\zeta'(2s)}{\zeta(2s)}=\sum_{n=1}^{\infty}\frac{\lambda(n)\Lambda(n)}{n^{s}},\label{eq:4.5}
\end{equation}
where $\Re(s)>1$.

For any $k\in\mathbb{N}$ and $\sigma>1$ let $\psi(k)=-\frac{\zeta(\sigma)}{\zeta^{\prime}(\sigma)}\frac{\Lambda(k)}{k^{\sigma}}>0$,
then by \eqref{eq:4.1}
\begin{equation}
\sum_{k=1}^{\infty}\psi(k)=-\frac{\zeta(\sigma)}{\zeta^{\prime}(\sigma)}\sum_{k=1}^{\infty}\frac{\Lambda(k)}{k^{\sigma}}=1.\label{eq:4.6}
\end{equation}
Now, let $\phi_{n}(x)=\frac{1}{x^{\lambda_{n}}}$ and 
\begin{equation}
a_{m,n}=-\frac{\zeta(\sigma)}{\zeta^{\prime}(\sigma)}\sum_{k=1}^{\infty}\frac{\Lambda(k)}{k^{(\lambda_{m}+\lambda_{n}+\sigma)}}=\frac{\zeta(\sigma)}{\zeta^{\prime}(\sigma)}\frac{\zeta^{\prime}(\sigma+\lambda_{m}+\lambda_{n})}{\zeta(\sigma+\lambda_{m}+\lambda_{n})},\label{eq:4.7}
\end{equation}
where $\left\{ \lambda_{n}\right\} _{n\in\mathbb{N}}$ is a sequence
of positive numbers such that $\lambda_{n}\ge n-1$. Then,
\begin{equation}
\begin{aligned} & \sup_{m\in\mathbb{N}}\sum_{n=1}^{\infty}\left|a_{m,n}\right|\le\left|\frac{\zeta(\sigma)}{\zeta^{\prime}(\sigma)}\right|\cdot\sup_{m\in\mathbb{N}}\sum_{k=2}^{\infty}\sum_{n=1}^{\infty}\frac{\Lambda(k)}{k^{m+n+\sigma-2}}\\
 & \le\left|\frac{\zeta(\sigma)}{\zeta^{\prime}(\sigma)}\right|\cdot\sup_{m\in\mathbb{N}}\sum_{k=2}^{\infty}\frac{\log(k)}{k^{m+\sigma-2}(k-1)}\le\left|\frac{\zeta(\sigma)}{\zeta^{\prime}(\sigma)}\right|\sum_{k=2}^{\infty}\frac{\log(k)}{k^{\sigma-1}(k-1)}<\infty.
\end{aligned}
\label{eq:4.8}
\end{equation}
For $t\in\mathbb{R}$ let $f(x)=a(x)x^{-it}$ where $a(x)$ is a completely
multiplicative function such that
\begin{equation}
A(s)=\sum_{n=1}^{\infty}\frac{a(n)}{n^{s}},\quad-\frac{A'(s)}{A(s)}=\sum_{n=1}^{\infty}\frac{a(n)\Lambda(n)}{n^{s}}.\label{eq:4.9}
\end{equation}
Then,
\begin{equation}
\sum_{k=1}^{\infty}\left|f(k)\right|^{2}\psi(k)=\sum_{k=1}^{\infty}\left|a(k)\right|^{2}\psi(k)\le\sum_{k=1}^{\infty}\psi(k)=1\label{eq:4.10}
\end{equation}
and 
\begin{equation}
\begin{aligned} & f_{n}=\sum_{k=1}^{\infty}\frac{f(k)}{k^{\lambda_{n}}}\psi(k)\\
 & =-\frac{\zeta(\sigma)}{\zeta^{\prime}(\sigma)}\sum_{k=1}^{\infty}\frac{a(k)\Lambda(k)}{k^{\lambda_{n}+s}}=\frac{\zeta(\sigma)}{\zeta^{\prime}(\sigma)}\frac{A'(s+\lambda_{n})}{A(s+\lambda_{n})}.
\end{aligned}
\label{eq:4.11}
\end{equation}
 Then by Corollary \ref{cor:5},
\begin{equation}
\sum_{n=1}^{\infty}\left|\frac{A'(s+\lambda_{n})}{A(s+\lambda_{n})}\right|^{2}\le\sup_{m\in\mathbb{N}}\sum_{n=1}^{\infty}\frac{\zeta^{\prime}(\sigma)\zeta^{\prime}(\sigma+\lambda_{m}+\lambda_{n})}{\zeta(\sigma)\zeta(\sigma+\lambda_{m}+\lambda_{n})}.\label{eq:4.12}
\end{equation}
In particular, for $a(n)=1$,
\begin{equation}
\sum_{n=1}^{\infty}\left|\frac{\zeta'(s+\lambda_{n})}{\zeta(s+\lambda_{n})}\right|^{2}\le\sup_{m\in\mathbb{N}}\sum_{n=1}^{\infty}\frac{\zeta^{\prime}(\sigma)\zeta^{\prime}(\sigma+\lambda_{m}+\lambda_{n})}{\zeta(\sigma)\zeta(\sigma+\lambda_{m}+\lambda_{n})},\label{eq:4.13}
\end{equation}
 for $a(n)=\lambda(n)$,
\begin{equation}
\sum_{n=1}^{\infty}\left|\frac{\zeta'(s+\lambda_{n})}{\zeta(s+\lambda_{n})}-2\frac{\zeta'(2s+2\lambda_{n})}{\zeta(2s+2\lambda_{n})}\right|^{2}\le\sup_{m\in\mathbb{N}}\sum_{n=1}^{\infty}\frac{\zeta^{\prime}(\sigma)\zeta^{\prime}(\sigma+\lambda_{m}+\lambda_{n})}{\zeta(\sigma)\zeta(\sigma+\lambda_{m}+\lambda_{n})},\label{eq:4.14}
\end{equation}
and $a(n)=\chi(n)$,
\begin{equation}
\sum_{n=1}^{\infty}\left|\frac{L'(s+\lambda_{n})}{L(s+\lambda_{n})}\right|^{2}\le\sup_{m\in\mathbb{N}}\sum_{n=1}^{\infty}\frac{\zeta^{\prime}(\sigma)\zeta^{\prime}(\sigma+\lambda_{m}+\lambda_{n})}{\zeta(\sigma)\zeta(\sigma+\lambda_{m}+\lambda_{n})}.\label{eq:4.15}
\end{equation}

\subsection{\textmd{Case: $\psi(k)=\frac{1-\delta_{k,1}}{k^{\sigma}(\zeta(\sigma)-1)}$}}

Let $\phi_{n}(x)$ and $\lambda_{n}>n-1$ be as in the last section,
then 

\begin{equation}
a_{m,n}=\frac{1}{\zeta(\sigma)-1}\sum_{k=2}^{\infty}\frac{1}{k^{\sigma+\lambda_{m}+\lambda_{n}}}=\frac{\zeta(\sigma+\lambda_{m}+\lambda_{n})-1}{\zeta(\sigma)-1}.\label{eq:4.16}
\end{equation}
 Then, 
\begin{equation}
\begin{aligned} & \sup_{m\in\mathbb{N}}\sum_{n=1}^{\infty}\left|a_{m,n}\right|\le\frac{1}{\zeta(\sigma)-1}\sup_{m\in\mathbb{N}}\sum_{k=2}^{\infty}\sum_{n=1}^{\infty}\frac{1}{k^{m+n+\sigma-2}}\\
 & \le\frac{1}{\zeta(\sigma)-1}\sup_{m\in\mathbb{N}}\sum_{k=2}^{\infty}\frac{1}{k^{m+\sigma-2}(k-1)}\le\frac{1}{\zeta(\sigma)-1}\sum_{k=2}^{\infty}\frac{1}{k^{\sigma-1}(k-1)}<\infty.
\end{aligned}
\label{eq:4.17}
\end{equation}
 For any $t\in\mathbb{R}$ let
\begin{equation}
f(x)=\frac{a(x)}{x^{it}},\label{eq:4.18}
\end{equation}
where $a(x)$ is any multiplicative function satisfying $\left|a(k)\right|\le1,\ k\in\mathbb{N}$.
Then for $s=\sigma+it$ we have
\begin{equation}
f_{n}=\sum_{k=1}^{\infty}f(k)\phi_{n}(k)\psi(k)=\frac{1}{\zeta(\sigma)-1}\sum_{k=2}^{\infty}\frac{a(k)}{k^{s+\lambda_{n}}}\label{eq:4.19}
\end{equation}
and
\begin{equation}
\sum_{k=1}^{\infty}\left|f(k)\right|^{2}\psi(k)=\frac{1}{\zeta(\sigma)-1}\sum_{k=2}^{\infty}\frac{\left|a(k)\right|^{2}}{k^{\sigma}}\le1.\label{eq:4.20}
\end{equation}
 Then by Corollary \ref{cor:5} we have
\begin{equation}
\sum_{n=1}^{\infty}\left|\sum_{k=2}^{\infty}\frac{a(k)}{k^{s+\lambda_{n}}}\right|^{2}\le\left(\sup_{m\in\mathbb{N}}\sum_{n=1}^{\infty}\left(\zeta(\sigma+\lambda_{m}+\lambda_{n})-1\right)\right)\sum_{k=2}^{\infty}\frac{\left|a(k)\right|^{2}}{k^{\sigma}}\cdot\label{eq:4.21}
\end{equation}
 Hence, for $a(k)=1,\chi(k),\mu(k),\mu(k)\chi(k),\lambda(k)$ we get
\begin{align}
 & \sum_{n=1}^{\infty}\left|\zeta(s+\lambda_{n})-1\right|^{2}\le\left(\zeta(\sigma)-1\right)\cdot\sup_{m\in\mathbb{N}}\sum_{n=1}^{\infty}\left(\zeta(\sigma+\lambda_{m}+\lambda_{n})-1\right),\label{eq:4.22}\\
 & \sum_{n=1}^{\infty}\left|L(s+\lambda_{n},\chi)-1\right|^{2}\le\left(\zeta(\sigma)-1\right)\cdot\sup_{m\in\mathbb{N}}\sum_{n=1}^{\infty}\left(\zeta(\sigma+\lambda_{m}+\lambda_{n})-1\right),\label{eq:4.23}\\
 & \sum_{n=1}^{\infty}\left|\frac{1}{\zeta(s+\lambda_{n})}-1\right|^{2}\le\left(\frac{\zeta(\sigma)}{\zeta(2\sigma)}-1\right)\cdot\sup_{m\in\mathbb{N}}\sum_{n=1}^{\infty}\left(\zeta(\sigma+\lambda_{m}+\lambda_{n})-1\right),\label{eq:4.24}\\
 & \sum_{n=1}^{\infty}\left|\frac{1}{L(s+\lambda_{n},\chi)}-1\right|^{2}\le\left(\frac{\zeta(\sigma)}{\zeta(2\sigma)}-1\right)\cdot\sup_{m\in\mathbb{N}}\sum_{n=1}^{\infty}\left(\zeta(\sigma+\lambda_{m}+\lambda_{n})-1\right),\label{eq:4.25}\\
 & \sum_{n=1}^{\infty}\left|\frac{\zeta(2s+2\lambda_{n})}{\zeta(s+\lambda_{n})}-1\right|^{2}\le\left(\zeta(\sigma)-1\right)\cdot\sup_{m\in\mathbb{N}}\sum_{n=1}^{\infty}\left(\zeta(\sigma+\lambda_{m}+\lambda_{n})-1\right).\label{eq:4.26}
\end{align}


\begin{thebibliography}{10}
\bibitem{Andrews} G. E. Andrews, R. A. Askey, and R. Roy, Special
Functions, Cambridge University Press, Cambridge, 1999.

\bibitem{AkhiezerGlazman}N. I. Akhiezer and I. M. Glazman, \emph{Theory
of Linear Operators in Hilbert Space}, Dover Publications, 1993. 

\bibitem{Apostol}T. M. Apostol, \emph{Introduction to analytic number
theory}, Springer, 1976.

\bibitem{Bellman}R. Bellman, Almost orthogonal series, Bulletin of
the American Mathematical Society, volume 50 (1944), issue 8, pp 517\textendash 519.

\bibitem{DLMF}NIST Digital Library of Mathematical Functions. http://dlmf.nist.gov/,
Release 1.1.3 of 2021-09-15. F. W. J. Olver, A. B. Olde Daalhuis,
D. W. Lozier, B. I. Schneider, R. F. Boisvert, C. W. Clark, B. R.
Miller, B. V. Saunders, H. S. Cohl, and M. A. McClain, eds.

\bibitem{HardyWright}G. H. Hardy and E. M. Wright, An Introduction
to the Theory of Numbers, 6th edition,Oxford University Press, 2008.

\bibitem{HornJohnson}R. A. Horn and C.R. Johnson, \emph{Matrix Analysis},
Cambridge University Press, 1990.

\bibitem{Ismail} M. E. H. Ismail, \emph{Classical and Quantum Orthogonal
Polynomials in one Variable}, paperback edition, Cambridge University
Press, Cambridge, 2009.

\bibitem{Oberhettinger2}F. Oberhettinger, \emph{Tables of Mellin
Transforms}, Springer, 1974.

\bibitem{Oberhettinger3}F. Oberhettinger and L. Badii, \emph{Tables
of Laplace Transforms}, Springer, 1973.

\bibitem{IsmailZhang1}M. E. H. Ismail and R. Zhang, Integral and
Series Representations of \$q\$-Polynomials and Functions: Part I,
, Analysis and Applications, Analysis and Applications, Vol.16, No.02,
(2018), 209\textendash 281.
\end{thebibliography}
\end{document}